\documentclass[12pt]{amsart}
\usepackage{amsmath,amsfonts,amsthm,amscd,amssymb,mathrsfs,amssymb}
\usepackage{stmaryrd}
\usepackage{graphicx}
\newtheorem{theorem}{Theorem}

\newtheorem{definition}[theorem]{Definition}

\newtheorem{remark}[theorem]{Remark}

\newcommand{\CC}{\mathbb{C}}

\newcommand{\RR}{\mathbb{R}}

\newcommand{\U}{{\rm{U}}}

\newcommand{\vphi}{\varphi}
\numberwithin{equation}{section}
\numberwithin{theorem}{section}
\numberwithin{table}{section}
\numberwithin{table}{section}
\begin{document}
\bibliographystyle{plain} 
\title[Scalar curvatures on compact Hermitian manifolds]{On the proportionality of\\ Chern and Riemannian scalar curvatures}
\author{Michael G. Dabkowski}
\address{Dept. of Mathematics, Lawrence Technological University, Southfield, 
MI,~48075}
\email{mdabkowsk@ltu.edu}
\author{Michael T. Lock}
\address{Dept. of Mathematics, University of Texas, Austin, 
TX, 78712}
\email{mlock@math.utexas.edu}
\thanks{The second author was partially supported by NSF Grant DMS-1148490}
\date{June 21, 2016}
\begin{abstract}
On a K\"ahler manifold there is a clear connection between the complex geometry and underlying Riemannian geometry.  In some ways, this can be used to characterize the K\"ahler condition.  While such a link is not so obvious in the non-K\"ahler setting, one can seek to understand extensions of these characterizations to general Hermitian manifolds.  This idea has been the subject of much study from the cohomological side, however, the focus here is to address such a question from the perspective of curvature relationships.  In particular, on compact manifolds the K\"ahler condition is characterized by the relationship that the Chern scalar curvature is equal to half the Riemannian scalar curvature.  What we study here is the existence, or lack thereof, of non-K\"ahler Hermitian metrics for which a more general proportionality relationship between these scalar curvatures holds.
\end{abstract}
\maketitle
\setcounter{tocdepth}{1}
\tableofcontents

\section{Introduction}
In studying the geometry of complex manifolds, one often seeks Hermitian metrics with certain ``desirable'' properties.   For instance, K\"ahler geometry provides a very deep and beautiful picture.  While this has long been an exciting and fruitful area of research, the K\"ahler condition is somewhat restrictive.
In the non-K\"ahler setting, unfortunately, a great deal of geometric structure and intuition is lost, and it is not yet completely clear what are the fundamental questions or geometric characteristics to investigate.  Much of the existing work in this direction focuses either on metrics which satisfy certain structural conditions that are a weakening of those found in K\"ahler geometry, or on Hermitian analogues of classical Riemannian questions concerning curvature.  For example, see \cite{Melvyn_Berger,Gauduchon_metric,Michelsohn,delRio_Simanca,Angella-Simone-Spotti}.
The work of this paper can be viewed as somewhere between these two ideas
as we study the weakening of a relationship between scalar curvatures that characterizes the K\"ahler condition.

In particular, we will be concerned with a relationship between the scalar curvatures resulting from what are, perhaps, the two most natural connections to consider in this setting - the Chern connection on the holomorphic tangent bundle, and the Levi-Civita connection on the underlying Riemannian manifold.  
The Riemannian and Chern scalar curvatures of a Hermitian manifold $(M,J,g)$ will be denoted by $S(g)$ and $S_C(g)$ respectively.  These are the scalar curvatures corresponding to the Levi-Civita and Chern connections.
Many fundamental questions in Riemannian geometry concern the Riemannian scalar curvature, which provides one of the most simple and intuitive geometric pictures of the underlying manifold as it arises in the asymptotic expansion of the volume of geodesic balls, see \cite{Besse}.  
The motivation for studying Chern scalar curvature here, as opposed to a variety of other complex curvature scalars, comes mainly from the naturalness of the connection and the resulting complex geometry, as well as it's interplay with the underlying real geometry.  It is worthwhile to note, though, that there has been much other recent interest in problems regarding Chern curvatures; for instance, in non-K\"ahler Calabi-Yau problems and the Chern Yamabe problem, see \cite{Tosatti_non-KahlerCY,Angella-Simone-Spotti}.

On a {\em compact} Hermitian manifold, the Chern and Riemannian scalar curvatures satisfy the relationship $S(g)=2\cdot S_C(g)$ if and only if the metric is K\"ahler \cite{Liu-Yang_2,Gauduchon_scalar}.  In fact, if $S(g)$ and $2\cdot S_C(g)$ are even equal in average over the manifold, then the metric must be K\"ahler.  Therefore, on a compact manifold, this relationship of scalar curvatures can be seen to characterize the K\"ahler condition.  
One natural question to ask is whether, on a non-compact manifold, there can exist non-K\"ahler Hermitian metrics for which the relationship of scalar curvatures $S(g)=2\cdot S_C(g)$ holds.  In \cite{Dabkowski-Lock_Klsc}, the authors examined this question for $\U(n)$-invariant Hermitian metrics on annuli in $\CC^n$, and showed that every conformal class admits a $1$-dimensional family (up to scale) of non-K\"ahler Hermitian metrics with the property that $S(g)=2\cdot S_C(g)$.  Furthermore, the authors classified the space of all such metrics on these annuli.

What we will study here is the existence of non-K\"ahler Hermitian metrics on compact
complex manifolds, which are always assumed to be {\em without} boundary, for which a generalization of the scalar curvature relationship that characterizes the K\"ahler condition in the compact setting holds.  Specifically, we ask whether there can exist Hermitian metrics on compact complex manifolds which posses the quality that $S(g)=2\kappa\cdot S_C(g)$ for some constant $\kappa$.  Such metrics would necessarily be non-K\"ahler for $\kappa\neq 1$.  It is interesting to note that if such metrics exist, then a simple and intuitive picture of the underlying geometry can be extracted from the Chern scalar curvature.  More specifically, we make the following definition.

\begin{definition}
{\em  A Hermitian metric $g$ on a complex manifold $(M,J)$ is said to be a  {\em scaled K\"ahler like scalar curvature (sKlsc) metric} if
\begin{align*}
S(g)=2\kappa \cdot S_C(g)
\end{align*}
for some constant $\kappa\neq0$.  The associated constant $\kappa$ is called the {\em scaling constant}.}
\end{definition}
\begin{remark}
{\em
We require that $\kappa\neq 0$ because, if $\kappa=0$, no intuition for the Chern scalar curvature would be derived from this relationship as the question is just that of existence for a Riemannian-scalar-flat metric. }
\end{remark}

For certain reasons, which will become apparent, it is interesting to study the existence of sKlsc metrics within a conformal class.  On a Hermitian manifold $(M, J,g)$, it is assumed that the complex structure is fixed and that any metric conformal to $g$ is Hermitian with respect to it.

\begin{align*}
&\text{{\bf{The sKlsc Problem:}} {\em In a Hermitian conformal structure  on a compact complex}}\\
&\text{{\em manifold, of complex dimension $n\geq 2$, does there exist an sKlsc representative?}}
\end{align*}

We are able to address this question in the case where the conformal class contains a balanced representative, which we denote by $g_b$.  This is because, when there is not a balanced representative, the operator we study becomes non-self-adjoint.  
The sKlsc problem differs significantly depending upon whether or not the balanced representative is actually a K\"ahler metric.  What we will see is that, when there is a K\"ahler representative, there are fewer existence results but clear uniqueness results, and, when there is not a K\"ahler representative, there are more broad existence results but less clear uniqueness results.  The nature of the sKlsc question also very much depends upon the Gauduchon degree of the conformal class.  In particular, the negative Gauduchon degree case depends upon the behavior of a relationship of scalar curvatures around an analytically degenerate instantiation of the problem.  Specifically, it splits into two cases, depending upon whether or not the the scalar curvatures at the balanced representative, $g_b$, satisfy
\begin{align}
\begin{split}
\label{negative_conditions}
&S(g_b)-2\Big(\frac{2n-1}{n}\Big)S_C(g_b)\geq 0\\
\text{with}&\\
&S(g_b)|_p-2\Big(\frac{2n-1}{n}\Big)S_C(g_b)|_p=0\phantom{=}\text{at least at one point $p\in M$}\\
\text{where}&\\
&S(g_b)|_p\neq0\phantom{ii}\text{and}\phantom{ii} S_C(g_b)|_p\neq0.
\end{split}
\end{align}


It is only possible to establish non-trivial existence results in for Hermitian conformal structures with non-zero Gauduchon degree, see Theorem \ref{zero_Gauduchon} below.  For the purpose of clarity of exposition, we give a broad picture of existence for sKlsc metrics for these cases in Theorem \ref{broad_theorem}.  Subsequently, in Section \ref{contains_Kahler_representative} and Section\ref{no_Kahler_representative}, a detailed account of these results is provided.

\begin{theorem}
\label{broad_theorem}
Let $[g]$ be a Hermitian conformal structure on a compact complex manifold of complex dimension $n\geq 2$, which contains a balanced representative $g_b$.  If the Gauduchon degree is
\vspace{.075in}
\begin{itemize}
\item positive, then there exists at least one sKlsc representative of $[g]$ and every such metric has scaling constant $\kappa\in (-\infty,1]$.
\vspace{.075in}
\item negative, and condition \eqref{negative_conditions}
\vspace{.075in}
\begin{enumerate}
\item is not satisfied, then there exists at least one sKlsc representative of $[g]$ and every such metric has scaling constant $\kappa\in [1,\infty)$.
\vspace{.075in}
\item
\label{is_satisfied}
is satisfied, then the only scaling constant for which there can possibly exist a non-K\"ahler sKlsc representative of $[g]$ is $\kappa=\frac{2n-1}{n}$.  Furthermore, unless $S(g_b)-2(\frac{2n-1}{n})S_C(g_b)$ vanishes at least twice on the manifold and $S_C(g_b)\leq 0$, there can be no such metric. 
\end{enumerate}
\end{itemize}
\end{theorem}

The proof of this result, and the corresponding more detailed results given in Section \ref{contains_Kahler_representative} and Section \ref{no_Kahler_representative}, begins with the construction and examination of a $1$-parameter family of second order semi-linear PDEs in the scaling constant $\kappa$.  Outside of $\kappa=\frac{2n-1}{n}$, the problem transforms into a family of eigenvalue problems for Schr\"odinger operators on compact manifolds.  The subsequent study then separates into several distinct regimes depending upon how the geometry of the balanced manifold behaves in the neighborhood of two critical scaling constants, $\kappa=1$ and $\kappa=\frac{2n-1}{n}$.  We employ our previous results from \cite{Dabkowski-Lock_Schrodinger} to study the eigenvalue behavior in the majority of these cases.  However, in one particular situation it is necessary to deal with a family of Schr\"odinger operators that has both a scaling and a warping of the potential function.  We examine this question separately, and in a general setting, in Section \ref{lack_of_stability}, and prove a surprising lack of stability in behavior.  There is however, one analytically degenerate case.  This is for negative Gauduchon degree and the scaling constant $\kappa=\frac{2n-1}{n}$.  For this scaling constant, the second order PDE degenerates into a first order non-linear autonomous Hamilton-Jacobi equation.  The existence of a classical solution to such an equation on a compact manifold is very rare.  It is common to instead seek a viscosity solution which in general has lower regularity.  Not only is this lower regularity not suitable for our purposes here, but, interesting, for the particular equation arising in this problem, the existence of a viscosity solution would be incredibly rare as well.

Theorem \ref{broad_theorem} is stated to include the possibility of K\"ahler metrics, which are clearly sKlsc, existing in the conformal class.  While this gives the most general picture of existence for these metrics, the true question to be dealt with is that for non-K\"ahler sKlsc metrics.  The specific results are detailed in Section \ref{contains_Kahler_representative} for conformal classes which contain a K\"ahler representative, and in Section \ref{no_Kahler_representative} for conformal classes which contain a balanced metric but no K\"ahler representative.

Interestingly, while sKlsc metrics are guaranteed to exist in almost every conformal class of non-zero Gauduchon degree, the sKlsc condition in a conformal class of zero Gauduchon degree is equivalent to the K\"ahler condition.  Specifically, we have the following result.

\begin{theorem}
\label{zero_Gauduchon}
Let $[g]$ be a Hermitian conformal structure on a compact complex manifold, of complex dimension $n\geq 2$, of zero Gauduchon degree.  There does not exist an sKlsc representative unless the conformal class contains a K\"ahler representative in which case that is the only sKlsc representative.
\end{theorem}

While Theorem \ref{zero_Gauduchon} addresses the sKlsc question completely for the case of zero Gauduchon degree, the techniques we have developed to study it in the case of non-zero Gauduchon degree require the existence of a balanced metric in the conformal class to eliminate the non-self-adjoint issue.  It would be interesting to study this question in full generality.

\subsection{When $[g]$ contains a K\"ahler representative}
\label{contains_Kahler_representative}
Here we seek to understand the existence of non-K\"ahler sKlsc metrics in conformal classes which contain a K\"ahler representative, $g_k$. 
Then, if such metrics do exist, what are the associated scaling constants, and for each such constant is the metric unique up to scale?  The answers depend significantly upon the behavior of the scalar curvature of the K\"ahler representative.  Since $S(g_k)=2\cdot S_C(g_k)$, the scalar curvature conditions will be stated simply in terms of $S_C(g_k)$.  Recall, by virtue of Theorem \ref{zero_Gauduchon}, it is only necessary to study the non-zero Gauduchon degree cases.
The sKlsc problem is addressed for negative and positive Gauduchon degrees in Theorem \ref{negative_Gauduchon_Kahler_specific} and Theorem \ref{positive_Gauduchon_Kahler_specific} respectively.

\begin{theorem}
\label{negative_Gauduchon_Kahler_specific}
Let $[g]$ be a Hermitian conformal structure of negative Gauduchon degree on a compact complex manifold of complex dimension $n\geq 2$.  If $[g]$ contains a K\"ahler representative, $g_k$, and
\begin{enumerate}
\item 
\label{sKlsc_existence}
$S_C(g_k)$ changes sign, then there are exactly two non-K\"ahler sKlsc representatives of $[g]$ up to scale.  These have scaling constants, $\kappa_1$ and $\kappa_2$, which lie in the respective regimes
\begin{align*}
1<\kappa_1<\frac{2n-1}{n}\phantom{=}\text{and}\phantom{=}\frac{2n-1}{n}<\kappa_2<\infty,
\end{align*}
and satisfy the relationship
\begin{align*}
(1-\kappa_1)(1-\kappa_2)=\frac{(n-1)^2}{n^2}.
\end{align*}
Furthermore, letting $e^{-2f_{\kappa_1}}$ and $e^{-2f_{\kappa_2}}$ denote the conformal factors by which the respective sKlsc metrics are obtained, normalized to unit length in $L^2$, then these conformal factors satisfy the relationship
\begin{align*}
e^{-2f_{\kappa_2}}=e^{-2(1-n)f_{\kappa_1}}.
\end{align*}
\item $S_C(g_k)\leq 0$, then the only scaling constant for which there can possibly exist a non-K\"ahler sKlsc representative of $[g]$ is $\kappa=\frac{2n-1}{n}$.  Furthermore, unless $S_C(g_k)$ vanishes at least twice on the manifold, no such representatives exist.
\end{enumerate}
\end{theorem}
\begin{remark}
{\em Regarding the existence result of Theorem \ref{negative_Gauduchon_Kahler_specific} part \eqref{sKlsc_existence}, it is important to note that while the possible range of the scaling constants $\kappa_1$ and $\kappa_2$ are the same for any compact K\"ahler manifolds satisfying these scalar curvature conditions, the actual values of these scaling constants almost certainly differs between such manifolds.  In fact, for each particular such manifold, the possible ranges for $\kappa_1$ and $\kappa_2$ can actually be bounded further away from $1$ and $\infty$ respectively.  However, we do not state these better bounds explicitly as their values have quite a complicated expression in terms of the total scalar curvature, the $L^{\infty}(M)$ norm of the scalar curvature, volume of the manifold and the Poincar\'e constant of the manifold.  For the interested reader, they can be obtained by substituting $\frac{(1-\kappa)(2n-1)(n-1)}{(n\kappa+1-2n)^2}S_C(g_k)$ for $\mathcal{V}$ in the inequality of Corollary \ref{t-schrodinger_cor_II} and solving for $\kappa$.}
\end{remark}

\begin{theorem}
\label{positive_Gauduchon_Kahler_specific}
Let $[g]$ be a Hermitian conformal structure of positive Gauduchon degree on a compact complex manifold of complex dimension $n\geq 2$.  If $[g]$ contains a K\"ahler representative, $g_k$, and
\begin{enumerate}
\item $S_C(g_k)$ changes sign, then the only scaling constants for which a non-K\"ahler sKlsc representative of $[g]$ can possibly exist are in the regime $(-\infty,1)$.  Furthermore, if
\begin{align*}
\frac{2n-1}{4n}\leq \frac{\int_M S_C(g_k)}{P||S_C(g_k)||_{\infty}\big(4Vol(M)||S_C(g_k)||_{\infty}+\int_MS_C(g_k)\big)},
\end{align*}
where $P$ is the Poincar\'e constant, then the sKlsc problem has no solution.
\item $S_C(g_k)\geq 0$, then no non-K\"ahler sKlsc representatives of $[g]$ exist.
\end{enumerate}
\end{theorem}

The same techniques used to prove Theorem \ref{negative_Gauduchon_Kahler_specific} part \eqref{sKlsc_existence} are not entirely applicable in this situation as there is a certain bound on the potential function of the Schr\"odinger operator formulation of the equation $S(g)=2\kappa\cdot S_C(g)$.  However, this does not preclude the existence of a conformal non-K\"ahler sKlsc metric for some $\kappa<1$.  In fact, there are certain examples for which we know such a metric must exist (See Section \ref{postive_gauduchon_sometimes} for more details).  We suspect, though, that there are some compact K\"ahler manifolds, that satisfy these scalar curvature conditions, for which no conformal non-K\"ahler sKlsc metric exists.  It would be interesting to investigate what further properties of the initial K\"ahler manifold this may depend upon.

\subsection{When $[g]$ does not contain a K\"ahler representative}
\label{no_Kahler_representative}
Here we seek to understand the existence, or lack thereof, for sKlsc metrics in conformal classes which do not contain a K\"ahler representative.  The question here differs significantly from the situation where there does exist a K\"ahler representative, addressed above in Section~\ref{contains_Kahler_representative}, in two primary ways -- the existence of a non-self-adjoint first order term, and a warping of the potential function in the Schr\"odinger operator formulation of the sKlsc equation.  The latter of these two issues impacts the case of negative Gauduchon degree as it leads to a surprising lack of stability for eigenvalues in an associated family of Schr\"odinger operators.  We study this phenomenon, in a general sense, in Section \ref{lack_of_stability}.  Since any sKlsc metric in a conformal class of zero Gauduchon degree must be K\"ahler, recall Theorem \ref{zero_Gauduchon}, we only need ask the sKlsc question here for negative and positive Gauduchon degrees.  These are addressed in Theorem \ref{negative_G_non_kahler} and Theorem \ref{positive_G_non_kahler} respectively.
It is interesting to contrast the results here with those corresponding in Section \ref{contains_Kahler_representative}, and comments are made accordingly.

\begin{theorem}
\label{negative_G_non_kahler}
Let $[g]$ be a Hermitian conformal structure of negative Gauduchon degree on a compact complex manifold of dimension $n\geq 2$.  If $[g]$ contains a balanced representative, $g_b$, but not a K\"ahler representative, and
\vspace{.075in}
\begin{enumerate}
\item $\displaystyle\int_M S(g_b)-2\Big(\frac{2n-1}{n}\Big)S_C(g_b)>0$, where $\displaystyle S(g_b)-2\Big(\frac{2n-1}{n}\Big)S_C(g_b)$
\vspace{.075in}
\begin{enumerate}
\item 
\label{possibly_both}
changes sign, then there exists at least one sKlsc representative which has scaling constant in the regime $(\frac{2n-1}{n},\infty)$.  
\vspace{.075in}

\item is non-negative, and the additional conditions of \eqref{negative_conditions} are not satisfied, then there exists at least one sKlsc representative and every such metric has scaling constant in the regime $(1, \frac{2n-1}{n})$.
\vspace{.075in}

\item 
\label{degenerate}is non-negative, and the additional conditions of \eqref{negative_conditions} are satisfied, then the only scaling constant for which a non-K\"ahler sKlsc representative can possibly exist is $\kappa=\frac{2n-1}{n}$.  Furthermore, unless $S(g_b)-2(\frac{2n-1}{n})S_C(g_b)$ vanishes at least twice on the manifold, no such representative exists. 
\end{enumerate}
\vspace{.075in}

\item $\displaystyle\int_M S(g_b)-2\Big(\frac{2n-1}{n}\Big)S_C(g_b)=0$, where $\displaystyle S(g_b)-2\Big(\frac{2n-1}{n}\Big)S_C(g_b)$
\vspace{.075in}
\begin{enumerate}
\item changes sign, then there exists at least one sKlsc representative and every such metric has scaling constant in the regime
$(\frac{2n-1}{n},\infty)$. 

\vspace{.075in}
\item does not change sign, then the balanced representative is the unique sKlsc representative of $[g]$ and it has scaling constant $\kappa=\frac{2n-1}{n}$.
\end{enumerate}

\vspace{.075in}

\item $\displaystyle\int_M S(g_b)-2\Big(\frac{2n-1}{n}\Big)S_C(g_b)<0$, then there exists at least one sKlsc representative and every such metric has scaling constant in the regime $(1,\frac{2n-1}{n})$. 
\end{enumerate}
\vspace{.075in}
\end{theorem}

\begin{remark}
{\em
In Theorem \ref{negative_G_non_kahler} \eqref{possibly_both}, while at least one sKlsc metric with scaling constant in the regime of $(\frac{2n-1}{n},\infty)$ is guaranteed to exist, in certain situations there may also exist sKlsc metrics with scaling constants in the regime $(1,\frac{2n-1}{n})$.  Also, recall the discussion following Theorem \ref{broad_theorem} regarding Theorem \ref{negative_G_non_kahler} \eqref{degenerate}.
}
\end{remark}

Considering Theorem \ref{negative_G_non_kahler} in light of the analogous result for the case that the conformal class contains a K\"ahler representative, Theorem \ref{negative_Gauduchon_Kahler_specific}, we see here that, while there is a broader range of existence results, there specificity of the analogous situation is lacking.  This is due to a wider possible range of behaviors of the warped potential function in the Schr\"odinger operator formulation of the sKlsc equation.

In the case of positive Gauduchon degree, the distinction between broadness of existence is even more extreme.  While in Theorem \ref{positive_Gauduchon_Kahler_specific} we see that existence of non-K\"ahler sKlsc metrics is very rare when the conformal class contains a K\"ahler representative, in Theorem \ref{positive_G_non_kahler} below we find that such a metric is always guaranteed to exist in this situation.

\begin{theorem}
\label{positive_G_non_kahler}
Let $[g]$ be a Hermitian conformal structure of positive Gauduchon degree on a compact complex manifold of dimension $n\geq 2$.  If $[g]$ contains a balanced representative, $g_b$, but not a K\"ahler representative, then there exists at least one sKlsc representative and every such metric has scaling constant in the regime~$(-\infty,1)$.  
\end{theorem}

\subsection{Acknowledgements}
The authors would like to thank Joseph Conlon and Lorenzo Sadun for many useful remarks, as well as John Gardner and Oliver Kelavan for many interesting conversations regarding the physical implications of this work.  Also, we would like to thank Jessica Dabkowski for putting up with us, which will be good practice for dealing with children.


\section{Background}
This section is composed to two parts which contain summaries of the necessary background in complex geometry and Schr\"odinger operators respectively.  Further references are given in each section.  

\subsection{Hermitian manifolds and scalar curvatures}
A brief description of the background necessary for sKlsc component of this work is provided in this section.  In particular, Hermitian manifolds, the Chern connection, its curvature, and the transformations of the Riemannian and Chern scalar curvatures under a conformal change are discussed.  For a more detailed account of these topics see \cite{Huybrechts,Besse,Liu-Yang_1,Angella-Simone-Spotti}.

A Riemannian metric $g$, on a complex manifold $(M,J)$, is said to be Hermitian if
$g(JX, JY)=g(X,Y)$
for any vector fields $X$ and $Y$.  The {\em associated $2$-form} of a Hermitian metric, $\omega\in\Lambda^{1,1}$, is defined by
$\omega(X,Y)=g(JX,Y)$, and can be written as
\begin{align}
\omega=\sqrt{-1}\sum_{i,j}g_{i\bar{j}}dz_i\wedge d\bar{z}_j,\phantom{=}\text{where}\phantom{=} g_{i\bar{j}}=g\Big(\frac{\partial}{\partial z_i},\frac{\partial}{\partial \bar{z}_j}\Big).
\end{align}
A Hermitian metric is called {\em K\"ahler} if its associated $2$-form is closed.

On a Hermitian manifold $(M, J, g)$, there exists a unique connection on the holomorphic tangent bundle $T^{1,0}M$, known as the {\em Chern connection}, which is compatible with the Hermitian metric and the complex structure.  The bundles $T^{1,0}M$ and $TM$ are naturally isomorphic, which allows for a comparison between the Chern connection and the Levi-Civita connection.  While both are compatible with the metric, the Chern connection does not necessarily induce a symmetric connection on $TM$ and the Levi-Civita connection does not necessarily induce a connection compatible with the complex structure.  If these connections do satisfy the additional condition of the other, they are equal since both are unique.  This occurs if and only if $g$ is K\"ahler.

While knowledge of Riemannian geometry is assumed here, many readers may not be familiar with the Chern curvatures tensors (i.e. those with respect to the Chern connection) in the non-K\"ahler setting.  Thus, a brief description is provided, see \cite{Huybrechts,Liu-Yang_2} for more detail.  Let $(M,J,g)$ be a Hermitian manifold, of complex dimension $n$, with associated $2$-form $\omega$.  The components of the full Chern curvature tensor 
on the Hermitian holomorphic vector bundle $(T^{1,0}M,g)$ are
\begin{align}
\Theta(g)_{i\bar{j}k\bar{l}}=-\frac{\partial^{2}g_{k\bar{l}}}{\partial z_i\partial \bar{z}_j}+g^{p\bar{q}}\frac{\partial g_{p\bar{l}}}{\partial \bar{z}_j}\frac{\partial g_{k\bar{q}}}{\partial z_i}.
\end{align}
The Chern-Ricci curvature,
\begin{align}
\Theta(g)_{i\bar{j}}=g^{k\bar{l}}\Theta(g)_{i\bar{j}k\bar{l}},
\end{align}
is obtained by tracing on the $k$ and $\bar{l}$ indices, and the Chern scalar curvature,
\begin{align}
S_C(g)=g^{i\bar{j}}\Theta(g)_{i\bar{j}},
\end{align}
by tracing further on the $i$ and $\bar{j}$ indices.
The Chern-Ricci curvature has the associated Chern-Ricci form given by
$\rho=\sqrt{-1}\sum_{i,j}\Theta(g)_{i\bar{j}}dz_i\wedge d\bar{z}_j=-\sqrt{-1}\partial \bar{\partial}\log \det (g_{i\bar{j}})$.  

The focus of this paper is to study the existence of Hermitian metrics on compact manifolds for which a certain desired relationship between the Chern and Riemannian scalar curvatures hold.
It was shown in \cite{Liu-Yang_2} that these are related by
\begin{align}
\label{S_SC_relationship}
S(g)=2S_C(g)+\langle\partial\partial^*\omega,\omega\rangle+\langle\overline{\partial}\overline{\partial}^*\omega,\omega\rangle-2|\partial^*\omega|^2-\frac{1}{2}|T(g)|^2,
\end{align}
where $T(g)$ is the torsion tensor.  A relationship between these scalar curvatures was also found in terms of Lee forms in \cite{Gauduchon_scalar}.  Extensive use will be made of the transformations of the Riemannian and Chern scalar curvatures under a conformal change of the metric which are recalled as follows.  Consider the conformal change $\tilde{g}=e^{2f}g$.  It is well-known, for a real $2n$-dimensional manifold, that the Riemannian scalar curvature of this conformal metric is given by
\begin{align}
\label{S_transform}
S(\tilde{g})=e^{2f}\Big(2(2n-1)\Delta f-2(2n-1)(n-1)|\nabla f|^2+S(g)\Big),
\end{align}
where $\Delta$ denotes the Laplace-Beltrami operator with sign convention so that the spectrum is non-positive \cite{Besse}.
The Chern scalar curvature of this conformal metric is given by
\begin{align}
\label{S_C_transform_1}
S_C(\tilde{g})&=e^{2f}\Big(n\Delta f-n\langle df,\theta\rangle_{\omega}+S_C(g)\Big),
\end{align}
where $\theta$ is the Lee form, or torsion $1$-form, associated to $\omega$, which is defined by $d\omega^{n-1}=\theta\wedge \omega^{n-1}$.  For a balanced Hermitian metric, i.e. a Hermitian metric for which the Lee form vanishes, the transformation of the Chern scalar curvature simplifies accordingly.
In particular, note that K\"ahler metric are necessarily balanced.
See \cite{Angella-Simone-Spotti} for an excellent reference regarding the conformal transformation of $S_C(g)$.

\subsection{Schr\"odinger operators on compact manifolds}
\label{schrodinger_background}
The Schr\"odinger equation arises in an astonishing number of both mathematical and physical problems, and while its origins lie in the study of quantum mechanics, see \cite{Schrodinger} and references therein, it has driven an incredible mathematical theory itself.

On a Riemannian manifold, the time independent Schr\"odinger equation is given by
\begin{align}
L_{\mathcal{V}}(\vphi)=-\Delta\vphi+\mathcal{V}\vphi=\lambda\vphi,
\end{align}
where $\Delta$ is the Laplace-Beltrami operator with non-positive spectrum, $\mathcal{V}$ is a smooth potential and $\lambda$ is a constant.  It is of both physical and mathematical interest to understand the negative eigenvalues ({\em bound states}) of the Schr\"odinger operator $L_{\mathcal{V}}$, in particular the lowest eigenvalue ({\em ground state}).  On compact manifolds, the lowest eigenvalue is found  by minimizing the Rayleigh quotient
\begin{align}
\label{Rayleigh}
\min_{0\not\equiv\vphi\in H^1}\frac{\int_M|\nabla\vphi|^2+\mathcal{V}\vphi^2}{\int_M\vphi^2}.
\end{align}
Many of the classical results fail to hold in the compact setting due to a failure of certain heat kernel estimates.  Recently, there has been some work on eigenvalue problems in this setting \cite{Grigor'yan-Netrusov-Yau,Grigor'yan-Nadirashvili-Sire,Dolbeault-Esteban-Laptev-Loss}.

On a compact manifold, the only positive eigenfunction is that associated to the lowest eigenvalue, which is unique up to scale.  For the work of this paper, it will be essential to understand when the lowest eigenvalue of such operators on compact manifolds is zero ({\em when the ground state has zero energy}).    It is relatively simple to ascertain the sign of the lowest eigenvalue for every potential $\mathcal{V}$ except when
\begin{align}
\begin{split}
\label{potential_conditions}
&\text{(i)\phantom{ii} $\mathcal{V}$ is smooth and changes sign on the manifold}\\
&\text{(ii)\phantom{i} $\int_M \mathcal{V}>0$,}
\end{split}
\end{align}
The question of sign for the lowest eigenvalue in this case is subtle, particularly that of if the lowest eigenvalue can be positive.
The authors developed the following estimate which will be used frequently throughout this work.
\begin{theorem}[\cite{Dabkowski-Lock_Schrodinger}, Theorem $1.1$]
\label{t-schrodinger_cor_II}
Let $(M,g)$ be a compact manifold, and consider the Schr\"odinger operator $L_{\mathcal{V}}=-\Delta+\mathcal{V}$ with the potential $\mathcal{V}$ satisfying \eqref{potential_conditions}.  If
\begin{align*}
\frac{P||\mathcal{V}||_{\infty}\big(4Vol(M)||\mathcal{V}||_{\infty}+\int_M\mathcal{V}\big)}{\int_M\mathcal{V}}\leq1,
\end{align*}
where $P$ is the Poincar\'e constant, then the lowest eigenvalue of $L_{\mathcal{V}}$
is strictly positive.
\end{theorem}

It is important to the work of this paper to understand when the lowest eigenvalue is zero.  Certainly Theorem \ref{t-schrodinger_cor_II} provides an obstruction to this, but 
even for such an operator $L_{\mathcal{V}}$ that does not satisfy the conditions of Theorem \ref{t-schrodinger_cor_II},
it is unlikely that zero is contained in the spectrum, and even less likely that it is the lowest eigenvalue.  In \cite{Dabkowski-Lock_Schrodinger}, the authors studied this lowest eigenvalue question under scalings of the manifold given a fixed potential.
Since the Laplace-Beltrami operator scales inversely to the scaling of the metric, the question is equivalent to one under scalings of the potential.  More specifically, the authors studied the following generalization of such a scaling.  On a compact manifold, let $\mathcal{V}$ be a fixed potential satisfying \eqref{potential_conditions},
and consider the $1$-parameter family positive scaling multipliers $f(t)$, where $f:\RR^+\rightarrow \RR^+$ is positive continuous function satisfying
\begin{align}
\begin{split}
\label{f_conditions}
&\text{(i)\phantom{ii}$\lim_{t\rightarrow 0}f(t)=0$}\\
&\text{(ii)\phantom{ii}$\lim_{t\rightarrow\infty}f(t)=\infty$}.
\end{split}
\end{align}
In this setting, the authors proved the following theorem.

\begin{theorem}[\cite{Dabkowski-Lock_Schrodinger}, Theorem $1.2$]
\label{Schrodinger_f_thm}
Let $(M,g)$ be a compact manifold and consider the $1$-parameter family of Schr\"odinger operators
\begin{align*}
L_{f(t)\mathcal{V}}=-\Delta+f(t)\mathcal{V},
\end{align*}
where $\mathcal{V}$ and $f(t)$ satisfy \eqref{potential_conditions} and \eqref{f_conditions} respectively.  There exists a $t^*>0$ for which zero is the lowest eigenvalue of $L_{f(t^*)\mathcal{V}}$.  Furthermore, if $f(t)$ is a strictly monotone increasing function, then there is a unique such $t^*$.
\end{theorem}

\begin{remark}
{\em
The lowest eigenvalue of $L_{f(t^*)\mathcal{V}}$ being zero is equivalent to the existence of a positive solution to $L_{f(t^*)\mathcal{V}}(\vphi)=0$, and that for each such $t^*$ this solution is unique up scale.
A positive solution to ${L_{f(t^*)\mathcal{V}}}(\vphi)=0$, for $t^*>0$, is necessarily non-constant. 
}
\end{remark}

\begin{remark}
\label{negative_eigenvalue}
{\em
In fact, it was shown that there are regimes over which the lowest eigenvalue is positive and regimes when it is negative, see \cite[Theorem $1.2$]{Dabkowski-Lock_Schrodinger} in its entirety.  Under such a scaling Theorem \ref{t-schrodinger_cor_II} holds to show that there exists a positive regime.  We see there is a negative regime by examining the Rayleigh quotient \eqref{Rayleigh} using a test function with compact support in the region of the manifold where $\mathcal{V}<0$.  Clearly, for $f(t)$ large enough this will be negative (See \cite{Dabkowski-Lock_Schrodinger}, Proposition $3.1$).
}
\end{remark}

The full generality of our work on the sKlsc problem here actually requires the analysis of a family of Schr\"odinger operators that have both a scaling and a warping of the potential function.  Specifically, families of operators of the form 
\begin{align}
-\Delta+f(t)\Big(\mathcal{V}_1+h(t)\mathcal{V}_2\Big)
\end{align}
will be studied, where $\mathcal{V}_1,\mathcal{V}_2\in C^{\infty}(M)$ are fixed potentials, and $f(t)$ and $h(t)$ are scaling and warping functions respectively.  Largely, the behavior of the eigenvalues of such families operators is stable and the above results, for when there is only a scaling of the potential, can be utilized in the analysis.  However, in a certain circumstance, which is encountered in the studying of sKlsc question, there is actually a loss of stability for the operators in question.
The lack of stability is not immediately obvious and requires additional analysis.  This is done in Section \ref{lack_of_stability}.

\section{The sKlsc equation}
\label{sKlsc_equation}
Here we understand the sKlsc relationship $S(g)-2\kappa\cdot S_C(g)$ as a $1$-parameter family of partial differential equations over the conformal class of the metric in the parameter of the scaling constant $\kappa$.  Next, certain restrictions on the set of admissible scaling constants depending upon the Gauduchon degree of the conformal class are established.  Subsequently, we make a change of variables, depending upon the parameter $\kappa$, which transforms the original family of equations into a $1$-parameter family of Schr\"odinger type equations which have both a scaling and a warping of the potential.  In most cases, the behavior of the lowest eigenvalues of this family of operators is stable.  However, in one particular case, which is germane to the sKlsc problem, it is not clear whether or not the behavior is stable.  We develop the necessary theory around this in Section~\ref{lack_of_stability}, and show that there is, in fact, a lack of stability in certain circumstances.

Let $(M,J,g)$ be a compact Hermitian manifold of complex dimension $n\geq 2$.  Consider the conformal metric $\tilde{g}=e^{-2f}g$, and recall that we seek a conformal metric such that
\begin{align}
\label{sKlsc_equation_2}
S(\tilde{g})-2\kappa\cdot S_C(\tilde{g})=0.
\end{align}
We will refer to \eqref{sKlsc_equation_2} as the {\em sKlsc equation}.
In this section, the sKlsc equation will be formulated as a partial differential equation in the conformal factor $f$.  Subsequently, the question of whether there are certain conditions that either restrict the possible range of scaling constants, $\kappa$, for which there is potentially a solution to \eqref{sKlsc_equation_2}, or provide an obstruction to the existence of any solution to the sKlsc problem altogether, is addressed.  Recall that our interest is in non-K\"ahler sKlsc metrics, which  are equivalent to sKlsc metrics with scaling constants $\kappa\neq 1$.  Thus, for the remainder of this section, it is assumed that $\kappa\neq 1$.

Using the transformations of the Chern and Riemannian scalar curvatures under the conformal change of the metric $\tilde{g}=e^{-2f}g$ for $f\in C^{\infty}(M)$, recall \eqref{S_transform} and \eqref{S_C_transform_1}, the sKlsc equation can be formulated as the semi-linear PDE
\begin{align}
\label{sKlsc_PDE_1}
(2n-1-n\kappa)\Delta f-(2n-1)(n-1)|\nabla f|^2+n\kappa\langle df, \theta\rangle_{\omega}+\frac{1}{2}\Big(S(g)-2\kappa \cdot S_C(g)\Big)=0.
\end{align}
As $\kappa$ is allowed to vary, what we actually obtain is a $1$-parameter family of second-order semi-linear PDEs in the scaling constant.  Notice that the corresponding elliptic operators are not self-adjoint, due to the $\langle df,\theta\rangle_{\omega}$ term, unless the Lee form $\theta$ vanishes in which case the reference metric, $g$, is necessarily balanced.

Observe, as follows, that it is straightforward to obtain certain simple restrictions on when it is possible for solutions to the sKlsc problem to exist.  As a reference metric, choose the unique (up to homotheties) Gauduchon metric in the conformal class, $g_G$.  Since $M$ is compact, and the Lee form of the Gauduchon metric is co-closed, the non-self-adjoint first order term vanishes under integration and we see that any solution $f$ must satisfy
\begin{align}
\label{sKlsc_equation_integral}
\int_M|\nabla f|^2=\frac{1}{2(2n-1)(n-1)}\int_MS(g_G)-2\kappa\cdot S_C(g_G).
\end{align}
From \eqref{S_SC_relationship}, notice that $\int_M S(g_G)-2S_C(g_G)<0$.  Thus, if there is one or more solutions to the sKlsc problem, one finds that the range of possible scaling constants is restricted as follows:
\begin{align}
\begin{split}
\label{kapparestrict}
\text{If the Gauduchon degree of the conformal class is }
\begin{cases}
\text{negative}\phantom{=}& \kappa>1\\
\text{positive}\phantom{=}& \kappa<1
\end{cases}.
\end{split}
\end{align}

\begin{remark}
{\em  In the case that the Gauduchon degree of the conformal class is zero, there are no solutions to the sKlsc problem other than K\"ahler metrics.  This is seen by rewriting the right hand side of \eqref{sKlsc_equation_integral} as 
\begin{align}
\begin{split}
\int_MS(g_G)-2\kappa\cdot S_C(g_G)=&(2-2\kappa)\int_MS_C(g_G)+\int_M S(g_G)-2S_C(g_G)\\
=&\int_M S(g_G)-2S_C(g_G)\leq 0,
\end{split}
\end{align}
since $\int_M S_C(g_G)=0$ for the Gauduchon metric.  Thus \eqref{sKlsc_equation_integral} can only hold if both sides vanish, in which case the initial metric must be K\"ahler.
}
\end{remark}

Notice, when $\kappa=\frac{2n-1}{n}$, the coefficient of the Laplacian term in  \eqref{sKlsc_PDE_1} vanishes and sKlsc equation degenerates into the first-order non-linear PDE
\begin{align}
\label{(2n-1)/n_equation}
|\nabla f|^2-\frac{1}{n-1}\langle df, \theta\rangle_{\omega}=\frac{S(g)-2(\frac{2n-1}{n})S_C(g)}{2(2n-1)(n-1)}.
\end{align}
Therefore, the study of the sKlsc equation separates into the following two cases:
\begin{align}
\begin{split}
\bullet&\phantom{i}\text{When $\kappa\neq\frac{2n-1}{n}$}\\
\bullet&\phantom{i}\text{When $\kappa=\frac{2n-1}{n}$.}
\end{split}
\end{align}
The scaling constant $\kappa=\frac{2n-1}{n}$ will be referred to as the {\em degenerate scaling constant}.  There are many obstructions to the existence of sKlsc metrics for this degenerate scaling constant.  In particular, if the conformal class contains a balanced representative, then the sKlsc problem for $\kappa=\frac{2n-1}{n}$ is obstructed unless $S_C\leq0$ and not identically zero.  This degenerate case will be discussed below in Section \ref{sKlsc_(2n-1)/n}.
For the remainder of this section, it is assumed that $\kappa\neq\frac{2n-1}{n}$, and the sKlsc equation \eqref{sKlsc_PDE_1} will be examined accordingly.

Because we are now assuming that $\kappa\neq\frac{2n-1}{n}$, the change of variables
\begin{align}
\label{conformal_factor_change}
e^{-2f}=\vphi^{\frac{2(2n-1-n\kappa)}{(2n-1)(n-1)}}
\end{align}
is admissible for positive $\vphi\in C^{\infty}(M)$.  This change of variables involves the scaling constant and, in fact, what we have is a $1$-parameter family, in the parameter $\kappa$, of these changes of variables corresponding to the family of PDEs in \eqref{sKlsc_PDE_1}.  Under this change, \eqref{sKlsc_PDE_1}, the formulation of he sKlsc equation as a semi-linear PDE, can be transformed into a perturbation of a Schr\"odinger equation by a first-order non-self-adjoint term
\begin{align}
-\Delta \vphi-\frac{n\kappa}{2n-1-n\kappa}\langle d\vphi,\theta\rangle_{\omega}+\frac{(2n-1)(n-1)}{2(n\kappa+1-2n)^2}\Big(S(g)-2\kappa\cdot S_C(g)\Big)\vphi=0.
\end{align}
Since the spectrum of non-self-adjoint operators is non-real in general, in the non-zero Gauduchon degree setting we restrict our attention to conformal classes which contain a balanced representative, which we denote by $g_b$, so that the Lee form vanishes and the sKlsc equation based at $g_b$ becomes the Schr\"odinger equation
\begin{align}
\label{sKlsc_PDE_2}
L_{\kappa}(\vphi)=-\Delta \vphi+\frac{(2n-1)(n-1)}{2(n\kappa+1-2n)^2}\Big(S(g_b)-2\kappa\cdot S_C(g_b)\Big)\vphi=0,
\end{align}
which has real spectrum bounded below.
Since the conformal factor \eqref{conformal_factor_change} cannot vanish, the sKlsc problem here becomes the question of whether there exists a scaling constant $\kappa\neq \frac{2n-1}{n}$ for which there is a positive solution to \eqref{sKlsc_PDE_2}.  Since the manifold is compact, recall from Section \ref{schrodinger_background}, that this is exactly the question of whether or not there exists some scaling constant $\kappa$ for which the operator associated to the corresponding Schr\"odinger equation \eqref{sKlsc_PDE_2} has zero as its lowest eigenvalue.  Thus, the sKlsc problem is transformed into a $1$-parameter family of eigenvalue problems.  

We will denote the lowest eigenvalue of the operator $L_{\kappa}$ by $\lambda_0(\kappa)$.
Observe that, if the conformal class contains a K\"ahler representative, $g_k$, then the sKlsc equation \eqref{sKlsc_PDE_2} based at the K\"ahler metric is just $-\Delta \vphi+\frac{(1-\kappa)(2n-1)(n-1)}{(n\kappa+1-2n)^2}S_C(g_k)\vphi=0$, since $S(g_k)=2S_C(g_k)$.  In this case, the $1$-parameter family of Scr\"odinger equations is determined by a scaling in $\kappa$ of the fixed potential $S_C(g_k)$, and the problem can be studied in a more systematic way via the theory developed in \cite{Dabkowski-Lock_Schrodinger}.  However, if the conformal class does not contain a K\"ahler representative, then not only is there a scaling in $\kappa$ on the potential coming from the coefficient $\frac{(2n-1)(n-1)}{(n\kappa+1-2n)^2}$, but there is also a warping of the potential by $\kappa$ in the $S(g)-2\kappa\cdot S_C(g)$ component since $S(g)\neq 2S_C(g)$ for all metrics in the confomal class.  While this warping leads to certain analytic complications, it also allows for proving certain existence results that are not possible to completely guarantee in the case where the conformal class contains a K\"ahler representative.

It is interesting to notice the similarity between the sKlsc equation \eqref{sKlsc_PDE_2}, particularly in the case that the conformal class contains a K\"ahler representative, and the Lichnerowicz-Weitzenb\"ock formula in spin geometry, see \cite{Lichnerowicz}.  Also, note that the change of variables \eqref{conformal_factor_change} is reminiscent of the Cole-Hopf transformation for Burgers' equation, see \cite{Evans}.

The remainder of the proof separates by the Gauduchon degree of the conformal class as the behavior of the sKlsc equation, and the nature of the sKlsc question itself, differs depending upon whether it is positive or negative.  The negative and positive Gauduchon degree cases are addressed below in Section \ref{negative_g_deg} and Section \ref{positive_g_deg} respectively.


\section{Schr\"odinger operators with scaled and warped potentials}
\label{lack_stability}

\label{lack_of_stability}
In Section~\ref{sKlsc_equation}, we saw that the sKlsc equation can be transformed into a $1$-parameter family of Schr\"odinger equations whose operators are of the form
\begin{align}
\label{warpedoperators}
L_{t}(\vphi)=-\Delta\vphi+f(t)\big(\mathcal{V}_1+h(t)\mathcal{V}_2\big)\vphi,
\end{align}
where $\mathcal{V}_1,\mathcal{V}_2:M\rightarrow \RR$ are fixed potential functions on the $M$, and $f,h:(a,b)\rightarrow\RR$ are functions in the parameter $t\in(a,b)\subset\RR$ which scale and warp the fixed potentials respectively. 
For the sKlsc equation $\mathcal{V}_1=S(g)$, $\mathcal{V}_2=S_C(g)$, $f=\frac{(2n-1)(n-1)}{2(nk+1-2n)^2}$ and $h=2\kappa$ as the parameter $\kappa$ ranges over the admissible range of scaling constants in each case.
The parameter $t$ is used in this section, as opposed to $\kappa$, to distinguish the work here from the specific sKlsc question.

While it is the sKlsc equation which led to the consideration of such families of operators, the lack of stability that we examine here is interesting in a more general setting and is thus addressed from the perspective of a compact Riemannian manifold of arbitrary dimension.  Certainly the behavior of the lowest eigenvalue of families of operators of the form \eqref{warpedoperators} depends upon qualities of the fixed potentials $\mathcal{V}_1$ and $\mathcal{V}_2$, as well as qualities of the scaling and warping functions $f(t)$ and $h(t)$.  While it is of course impossible to characterize the behavior of these eigenvalues for completely arbitrary such functions, given some reasonable assumptions of their qualities, one can often do this using standard analysis of the Rayleigh quotient in some cases or the work of the authors in \cite{Dabkowski-Lock_Schrodinger}, mentioned in Theorem \ref{Schrodinger_f_thm} above.

However, there are several cases of very reasonable assumptions on the functions in the operators \eqref{warpedoperators} for which the behavior of the lowest eigenvalue cannot be easily classified.  One such case arises in our study of the sKlsc question, and is addressed in \eqref{not_stable_sKlsc} of Section \ref{no_K_rep_neg} below.  Many cases of the general problem  actually reduce to the same thing analytically, and what this difficulty amounts to is a lack of stability in the behavior of the lowest eigenvalue.  For simplicity, we state this analytic situation in terms that can be directly applied to the sKlsc problem.

Let $L_t(\cdot)$, where $a<t<b$, be a $1$-parameter family of operators of the form \eqref{warpedoperators} on a Riemannian manifold $(M,g)$, and denote the lowest eigenvalue by $\lambda_0(t)$.  Assume that the scaling and warping functions
\begin{align}
\begin{split}
\label{scale_warp_conditions}
f,h:&(a,b)\subset\RR\rightarrow \RR^+\phantom{=}\text{are positive, continuous and satisfy}\\
&\text{(i)}\phantom{\text{iii}}\lim_{t\rightarrow a}f(t)=C_{f(a)}<\infty\phantom{=}\text{and} \phantom{=}  \lim_{t\rightarrow b}f(t)=\infty\\
&\text{(ii)}\phantom{\text{ii}}\lim_{t\rightarrow a}h(t)=C_{h(a)}<\infty\phantom{=}\text{and}\phantom{=}\lim_{t\rightarrow b}h(t)=C_{h(b)}<\infty\\
&\text{(iii)}\phantom{\text{i}}\text{$h(t)$ is strictly monotone increasing on the interval $(a,b)$},
\end{split}
\end{align}
and that the fixed potentials $\mathcal{V}_1,\mathcal{V}_2\in C^{\infty}(M)$ satisfy
\begin{align}
\begin{split}
\label{fixed_potentials}
&\text{(i)}\phantom{\text{ii}}\mathcal{V}_1\leq \mathcal{V}_2\\
&\text{(ii)}\phantom{\text{i}}\int_M \mathcal{V}_1+C_{h(a)}\mathcal{V}_2<0.
\end{split}
\end{align}

From \eqref{fixed_potentials} part (ii) and the family of Rayleigh quotients associated to these operators \eqref{Rayleigh}, it is clear that for a range of $t$ close enough to $a$, the lowest eigenvalue of $L_t$ is strictly negative.  
If $\mathcal{V}_1+C_{h(b)}\mathcal{V}_2>0$ everywhere on the manifold, then from \eqref{Rayleigh} the lowest eigenvalue for a range of $t$ close enough to $b$ is strictly positive, and thus by continuity there exists a $t^*\in (a,b)$ for which $\lambda_0(t^*)=0$.  

Now, what we ask is the following:  What if $\mathcal{V}_1+C_{h(b)}\mathcal{V}_2\geq0$
with equality at least at one point? (It is, of course, assumed this is not identically zero.)    In the case that at every point $p\in M$ where $\mathcal{V}_1(p)+C_{h(b)}\mathcal{V}_2(p)=0$, the fixed potentials both vanish themselves, i.e. $\mathcal{V}_1(p)=\mathcal{V}_2(p)=0$, then the same argument as when $\mathcal{V}_1+C_{h(b)}\mathcal{V}_2>0$ holds to show that the behavior of the lowest eigenvalue is the same as in that situation.  

However, if at even one such point $p\in M$ this is not the case, in other words $\mathcal{V}_1(p)+C_{h(b)}\mathcal{V}_2(p)=0$ but neither $\mathcal{V}_1(p)$ nor $\mathcal{V}_2(p)$ vanish, then for any $\epsilon>0$
\begin{align}
\label{goesnegative}
\mathcal{V}_1(p)+h(b-\epsilon)\mathcal{V}_2(p)<0
\end{align}
since $h(t)$ is strictly monotone increasing and the fixed potentials satisfy \eqref{fixed_potentials}.
Although this now changes sign, for small enough such $\epsilon>0$
\begin{align}
\label{stillhaspositiveintegral}
\int_M\mathcal{V}_1+h(b-\epsilon)\mathcal{V}_2>0,
\end{align}
and therefore, by Remark \ref{negative_eigenvalue}, the lowest eigenvalue $\lambda_0(b-\epsilon)<0$.  Thus, it is not possible to guarantee the existence of a zero lowest eigenvalue using continuity and perturbative methods.  These perturbative methods, though, do not allow one to understand or exercise any control over the behavior of the lowest eigenvalue for $t\in(a,b)$ away from the ends of this interval.  Furthermore, there is no increasing condition on $f(t)$ and a scaling function which shrinks the entire potential away from the endpoints of $(a,b)$ is certainly admissible.  Thus, perhaps there are zero and positive lowest eigenvalues there.    We will show that this can never happen, and that $\lambda_0(t)<0$ for all $t\in(a,b)$.

\begin{theorem}
\label{scaled_warped_thm}
On a compact Riemannian manifold consider the $1$-parameter family of Schr\"odinger operators
\begin{align*}
L_t=-\Delta+f(t)\big(\mathcal{V}_1+h(t)\mathcal{V}_2\big),
\end{align*}
for $t\in(a,b)$, where the potentials satisfy \eqref{scale_warp_conditions} and \eqref{fixed_potentials}.  If $\mathcal{V}_1+C_{h(b)}\mathcal{V}_2\geq0$
and there exists a point $p\in M$ at which
\begin{align*}
\mathcal{V}_1(p&)+C_{h(b)}\mathcal{V}_2(p)=0\\
\text{where}\phantom{====}\\
\mathcal{V}_1(p)&\neq0\phantom{ii}\text{and}\phantom{ii} \mathcal{V}_2(p)\neq0, 
\end{align*}
then the lowest eigenvalue of $L_t$ is strictly negative for all $t\in(a,b)$.
\end{theorem}

\begin{proof}
The proof of this amounts to showing a lack of stability in the behavior.  Recall, from the above discussion, that the lowest eigenvalues for $t$ close enough to $a$ or $b$ are negative.  However, the perturbative arguments used to see this do not extend to the interior of this regime.  Consider the $2$-parameter family of operators
\begin{align}
L_{(t,s)}=-\Delta+f(t)\Big(\mathcal{V}_1+h(t)\big(\mathcal{V}_2+s\big)\Big),
\end{align}
where $\mathcal{V}_1,\mathcal{V}_2$ and $f(t),g(t)$ satisfy the conditions of the theorem, and $s$ is a constant ranging over $-\infty<s<\infty$.
Denote the lowest eigenvalue of $L_{(t,s)}$ by $\lambda_0(t,s)$ and the associated eigenfunction by $\vphi_{(t,s)}$.  We can assume $\vphi_{(t,s)}>0$ and that $||\vphi_{(t,s)}||_2=1$ so these eigenfunctions are unique.

We begin by examining how the behavior of the lowest eigenvalue changes with $s$.
Observe that, since
\begin{align}
\begin{split}
\lambda_0(t,s)\vphi_{(t,s)}=&-\Delta\vphi_{(t,s)}+f(t)\Big(\mathcal{V}_1+h(t)\big(\mathcal{V}_2+s\big)\Big)\vphi_{(t,s)}\\
=&\Big[-\Delta\vphi_{(t,s)}+f(t)\Big(\mathcal{V}_1+h(t)\mathcal{V}_2\Big)\vphi_{(t,s)}\Big]+h(t)f(t)s\cdot\vphi_{(t,s)}\\
=&\Big(\lambda_0(t,0)+h(t)f(t)s\Big)\vphi_{(t,s)},
\end{split}
\end{align}
for a fixed $t\in(a,b)$ the eigenfunctions corresponding to the lowest eigenvalues for all $s$ are in fact equal, i.e. for any $s_1,s_2\in(-\infty,\infty)$ the associated eigenfunctions $\vphi_{(t,s_1)}=\vphi_{(t,s_2)}$, and furthermore that  the lowest eigenvalues are related by
\begin{align}
\label{eigendifference}
\lambda_0(t,s_1)-\lambda_0(t,s_2)=f(t)h(t)(s_1-s_2).
\end{align}
Therefore, since $f(t)$ and $h(t)$ are positive,
\begin{align}
\label{lambda_derivative}
\frac{\partial}{\partial s}\big(\lambda_0(t,s)\big)>0.
\end{align}

Now, for any $s>0$,
\begin{align}
\mathcal{V}_1+C_{h(b)}\big(\mathcal{V}_2+s\big)>0.
\end{align}
Therefore there exists a range of $\epsilon>0$ for which $\lambda_0(b-\epsilon,s)>0$, which in turn guarantees the existence of some $t\in(a,b)$ at which the lowest eigenvalue is zero.  We will see eventually that there is a unique such $t$, but at this point there is still the possibility of several, and we denote the greatest of these by $t^*(s)$.  In other words, for any $s>0$ the lowest eigenvalue $\lambda_0(t^*(s),s)=0$ and $t^*(s)$ is the greatest value in $(a,b)$ for which this is true.  From \eqref{lambda_derivative} we see, while $s>0$, that 
\begin{align}
\frac{d}{d s}\big(t^*(s)\big)<0,
\end{align}
so as $s$ increases the largest value of $t\in(a,b)$ for which the lowest eigenvalue is zero decreases.  In a sense, what we find is that these $t^*(s)$ are foliating the regime $(a,b)$.

Finally, suppose that there exists a $t^*(0)\in(a,b)$ for which the lowest eigenvalue of $L_{(t,0)}$ is zero. Notice that $L_{(t,0)}$ is precisely the $1$-parameter family of Schr\"odinger operators in Theorem \ref{scaled_warped_thm}.  Then, since $(a,b)$ is follated by the $t^*(s)$ for $s>0$, there must be some particular $s>0$ for which $\lambda_0(t^*(0),0)=\lambda_0(t^*(s),s)=0$ and $t^*(0)=t^*(s)$.  However, from \eqref{eigendifference}, we see that this cannot occur.

\end{proof}

\section{Negative Gauduchon degree case}
\label{negative_g_deg}
Recall, from \eqref{kapparestrict}, that the admissible regime of scaling constants when the conformal class has negative Gauduchon degree is $\kappa\in (1,\infty)$.  The degenerate scaling constant $\kappa=\frac{2n-1}{n}$ separates this into two disjoint intervals $(1,\frac{2n-1}{n})\sqcup(\frac{2n-1}{n},\infty)$ and, in turn,
the sKlsc question here separates into two problems by restricting the sKlsc equation \eqref{sKlsc_PDE_2} to each of these intervals.
The study of these questions further separates depending upon whether the conformal class does or does not contain a K\"ahler representative.  Interestingly, the difference primarily lies in the range of scaling constants $\kappa\in(1,\frac{2n-1}{n})$.  The sKlsc question is answered in the non-degenerate setting in Section \ref{no_K_rep_neg} and Section \ref{exists_K_rep_neg}, and is addressed for the degenerate scaling constant in Section \ref{sKlsc_(2n-1)/n}.

\subsection{When $[g]$ does not contain a K\"ahler representative}
\label{no_K_rep_neg}
Consider, first, the sKlsc question in the case that there exists a balanced representative, $g_b$, of the conformal class, but no K\"ahler representative.  Base the family of sKlsc equations at this balanced metric, so the non-self-adjoint term vanish, and they take the form of \eqref{sKlsc_PDE_2}.  The result makes use of the continuity of the lowest eigenvalue of the family of operators $L_{\kappa}$ and relies upon a detailed analysis of the behavior of these eigenvalues throughout the admissible regime of scaling constants $(1,\frac{2n-1}{n})\sqcup(\frac{2n-1}{n},\infty)$, most importantly, around the ends of these open intervals.  The behavior is straightforward to understand for $\kappa$ ``close enough'' to $1$ and $\infty$, and is discussed forthwith.  However,
the analysis is quite subtle for $\kappa$ in a neighborhood of the degenerate scaling constant.

For $\kappa$ near $1$, since $\int_M S(g_b)< 2\int_M S_C(g_b)<0$, there exists an $\epsilon^*>0$ small enough so that, for all $\epsilon\in(0,\epsilon^*)$, the perturbation
\begin{align*}
 \int_MS(g_b)-2(1+\epsilon) S_C(g_b)<0.
 \end{align*}
Therefore, the lowest eigenvalue of the operators $L_{1+\epsilon}$ satisfy $\lambda_0(1+\epsilon)<0$.
Similarly, any $\kappa>>\frac{2n-1}{n}$ large enough ensures that $\int_M S(g_b)-2\kappa\cdot S_C(g_b)>0$.  Then, since $\int_M S(g_b)-2\kappa\cdot S_C(g_b)$ is an increasing function of $\kappa$ and 
\begin{align}
\lim_{\kappa\mapsto\infty}\frac{(2n-1)(n-1)}{2(n\kappa+1-2n)^2}\Big(S(g_b)-2\kappa \cdot S_C(g_b)\Big)=0,
\end{align}
there must exist some range of scaling constants $(\kappa^*,\infty)\subset(\frac{2n-1}{n},\infty)$ for which the operators $L_{\kappa}$ have lowest eigenvalues $\lambda_0(\kappa)>0$ for all $\kappa\in(\kappa^*,\infty)$.  This follows from Theorem \ref{Schrodinger_f_thm} and Remark \ref{negative_eigenvalue} if $S_C(g_b)$ changes sign on the manifold, and directly from examining the Rayleigh quotient if $S_C(g_b)$ is non-positive.

The behavior of the lowest eigenvalues for the operators $L_{\kappa}$ around the degenerate scaling constant $\kappa=\frac{2n-1}{n}$ depends upon certain qualities of $S(g_b)-2(\frac{2n-1}{n}) S_C(g_b)$ as a function on the manifold.  The question first separates into cases depending upon whether $\int_MS(g_b)-2(\frac{2n-1}{n})S_C(g_b)$ is positive, zero or negative.
\begin{enumerate}
\item
When $\int_MS(g_b)-2(\frac{2n-1}{n})S_C(g_b)>0$, small perturbations of the scaling constant preserve positivity, i.e. for $\epsilon>0$ small enough 
\begin{align*}
\int_MS(g_b)-2\Big(\frac{2n-1}{n}\pm\epsilon\Big)S_C(g_b)>0.
\end{align*}
The behavior of the lowest eigenvalue separates this case into three subcases depending upon if $S-2(\frac{2n-1}{n})S_C$ changes sign, is strictly positive or is non-negative and equals zero at least once on the manifold.\\

\begin{enumerate}
\item
\label{critical_sign_change} If $S(g_b)-2(\frac{2n-1}{n})S_C(g_b)$ changes sign on the manifold, then for a range of $\epsilon>0$ small enough, $S(g_b)-2(\frac{2n-1}{n}\pm\epsilon)S_C(g_b)$ changes sign on the manifold as well and, from Remark \ref{negative_eigenvalue}, we see that $\lambda_0(\frac{2n-1}{n}\pm\epsilon)<0$.  Thus, recalling $\lambda_0(\kappa>>\frac{2n-1}{n})>0$ for $\kappa$ large enough,
there is a scaling constant  $\kappa\in(\frac{2n-1}{n},\infty)$ for which a solution to the sKlsc problem exists.
\begin{remark}
{\em In the regime $(1,\frac{2n-1}{n})$ a solution is not guaranteed to exist since the lowest eigenvalues for scaling constants close to $1$ and $\frac{2n-1}{n}$ are negative.  However, it is certainly possible for there to be a solution in this regime as, for $\epsilon_1>0$ and $\epsilon_2>0$ small enough, $\lambda_0(1+\epsilon_1)<0$ due to the fact that $\int_M S(g_b)-2(1+\epsilon_1)S_C(g_b)<0$ for , while $\lambda_0(\frac{2n-1}{n}-\epsilon_2)<0$ follows from Remark \ref{negative_eigenvalue}
because $S(g_b)-2(\frac{2n-1}{n}-\epsilon_2)S_C(g_b)$ changes sign while $\int_M S(g_b)-2(1+\epsilon_2)S_C(g_b)>0$, and $\lim_{\epsilon_2\mapsto 0}\frac{1}{(n(\frac{2n-1}{n}-\epsilon_2)+1-2n)^2}=\infty$.}
\end{remark}

\item
\label{strict_critical_val}
If $S(g_b)-2(\frac{2n-1}{n})S_C(g_b)>0$, then for a range of $\epsilon>0$ small enough, 
\begin{align*}
S(g_b)-2\Big(\frac{2n-1}{n}\pm \epsilon\Big)S_C(g_b)>0
\end{align*}
as well, so $\lambda_0(\frac{2n-1}{n}\pm\epsilon)>0$.  Thus, recalling that $\lambda_0(1+\epsilon)<0$ for $\epsilon>0$ small enough, there is a scaling constant  $\kappa\in(1,\frac{2n-1}{n})$ for which a solution to the sKlsc problem exists.  
However, there are no solution to the sKlsc problem for scaling constants in the regime $(\frac{2n-1}{n},\infty)$ since $S(g_b)-2\kappa\cdot S_C(g_b)$ is an increasing function, and therefore positive, on this entire range.\\

\item 
\label{not_stable_sKlsc}
Suppose $S(g_b)-2(\frac{2n-1}{n})S_C(g_b)\geq0$, with $S(g_b)|_p-2(\frac{2n-1}{n})S_C(g_b)|_p=0$ at least at one point $p\in M$.  If at each such point $p\in M$, this is zero because $S(g_b)|_p=S_C(g_b)|_p=0$, the same argument as in \eqref{strict_critical_val} above shows that there is a scaling constant  $\kappa\in(1,\frac{2n-1}{n})$ for which a solution to the sKlsc problem exists, and that there are no solution to the sKlsc problem for scaling constants in the regime $(\frac{2n-1}{n},\infty)$.  However, if at any point $p\in M$ at which this vanishes neither of the scalar curvatures vanish, in other words
\begin{align*}
S(g_b)|_p-2\Big(&\frac{2n-1}{n}\Big)S_C(g_b)|_p=0\\
\text{but}\phantom{========}&\\
S(g_b)|_p\neq0&\phantom{ii}\text{and}\phantom{ii}S_C(g_b)|_p\neq0,
\end{align*}
 then there are no solutions to the sKlsc problem for scaling constants in the union of the regimes $ (1,\frac{2n-1}{n})\sqcup(\frac{2n-1}{n},\infty)$.  This is seen for the regime $(\frac{2n-1}{n},\infty)$ since $S(g_b)-2\kappa\cdot S_C(g_b)$ is an increasing function of $\kappa$, and therefore non-negative on this entire range, so $\lambda_0(\kappa)>0$ here.  For the regime of scaling constants $(1,\frac{2n-1}{n})$ the non-existence follows immediately from Theorem \ref{scaled_warped_thm}, which concerns the lack of stability for Schr\"odinger operators with such a warping and scaling of the potential.\\
\end{enumerate}

\item When $\int_MS(g_b)-2(\frac{2n-1}{n})S_C(g_b)=0$, then
\begin{align*}
\int_M S(g_b)-2\kappa \cdot S_C(g_b)\phantom{i}
\begin{cases}
<0\phantom{=}\text{for $\kappa<\frac{2n-1}{n}$}\\
>0\phantom{=}\text{for $\kappa>\frac{2n-1}{n}$}
\end{cases}.
\end{align*}
Thus, there are no solutions to the sKlsc problem for $\kappa\in(1,\frac{2n-1}{n})$ as $\lambda_0(\kappa)<0$ over this entire regime.  The question of existence for $\kappa\in(\frac{2n-1}{n},\infty)$ depends upon whether or not $S(g_b)-2(\frac{2n-1}{n})S_C(g_b)$ changes sign.\\

\begin{enumerate}
\item If $S(g_b)-2(\frac{2n-1}{n})S_C(g_b)$ changes sign, by the same argument as in \eqref{critical_sign_change}, there is a scaling constant $\kappa\in(\frac{2n-1}{n},\infty)$ for which a solution to the sKlsc problem exists.\\

\item  If $S(g_b)-2(\frac{2n-1}{n})S_C(g_b)$ does not change sign, then it is necessarily identically zero
since it integrates to zero.  Therefore, the balanced metric is an sKlsc metric with $\kappa=\frac{2n-1}{n}$.  Furthermore, this is the unique sKlsc metric, up to scale, in the conformal class.\\
\end{enumerate}

\item When $\int_MS(g_b)-2(\frac{2n-1}{n})S_C(g_b)<0$, for a range of $\epsilon>0$ small enough
\begin{align*}
\int_M S(g_b)-2\Big(\frac{2n-1}{n}+\epsilon\Big)S_C(g_b)<0
\end{align*}
as well, so $\lambda_0(\frac{2n-1}{n}+\epsilon)<0$.  Thus, recalling that $\lambda_0(\kappa>>\frac{2n-1}{n})>0$, there is a scaling constant  $\kappa\in (\frac{2n-1}{n},\infty)$ for which a solution to the sKlsc problem exists. However, there is no solution for any $\kappa\in (1,\frac{2n-1}{n})$ as $\int_M S(g_b)-2\kappa\cdot S_C(g_b)<0$, so $\lambda_0(\kappa)<0$ for all scaling constants in this regime.  \\
\end{enumerate}

\subsection{When $[g]$ contains a K\"ahler representative}
\label{exists_K_rep_neg}
In the case that there does exist a K\"ahler representative, $g_k$, of the conformal class, recall that the family of sKlsc equations based at this metric simplifies to 
$-\Delta \vphi+\frac{(1-\kappa)(2n-1)(n-1)}{(n\kappa+1-2n)^2}S_C(g_k)\cdot \vphi=0$, since for a K\"ahler metric $S(g_k)=2S_C(g_k)$.  This can be viewed as a $1$-parameter family of Schr\"odinger equations with a scaling of a fixed potential, $S_C(g_k)$, in the parameter $\kappa$.  In this situation, the sKlsc problem is exactly the same question for both of the regimes $\kappa_1\in(1,\frac{2n-1}{n})$ and $\kappa_2\in(\frac{2n-1}{n},\infty)$ as there is a bijective correspondence between given by the operators restricted to these given by setting 
\begin{align}
\label{F_1=F_2}
\frac{(\kappa_1-1)(2n-1)(n-1)}{(n\kappa_1+1-2n)^2}=\frac{(\kappa_2-1)(2n-1)(n-1)}{(n\kappa_2+1-2n)^2},
\end{align}
and solving to extract the relationship
\begin{align}
\label{k_1,k_2}
(\kappa_1-1)(\kappa_2-1)=\frac{(n-1)^2}{n^2}
\end{align}
between $\kappa_1$ and $\kappa_2$ in their respective domains.  Then, solving for either, observe that the derivative with respect to the other is nonvanishing in the appropriate regime.  Therefore, a solution to the problem in one regime guarantees a solution in the other since the operators $L_{\kappa_1}=L_{\kappa_2}$ for each unique pair of scaling constants $\kappa_1\in(1,\frac{2n-1}{n})$ and $\kappa_2\in(\frac{2n-1}{n},\infty)$ that satisfy \eqref{k_1,k_2}.  It is easy to check that $L_{\kappa_2}=-\Delta +\frac{(1-\kappa_2)(2n-1)(n-1)}{(n\kappa_2+1-2n)^2}S_C(g_k)$ satisfies the conditions of Theorem \ref{Schrodinger_f_thm}, and thus the existence and uniqueness results of Theorem \ref{negative_Gauduchon_Kahler_specific} part \eqref{sKlsc_existence} follow accordingly.    The relationship between the conformal factors follows from \eqref{k_1,k_2} and \eqref{conformal_factor_change}.

\begin{remark}
{\em Although it is the same positive function $\vphi$ that provides the solution to the sKlsc equation in both the regimes $(1,\frac{2n-1}{n})$ and $(\frac{2n-1}{n},\infty)$, and this function is unique up to scale, it is important to note that the actual conformal factors \eqref{conformal_factor_change} used to obtain the sKlsc metric for each of the scaling constants, $\kappa_1$ and $\kappa_2$, are different powers of $\vphi$, and thus the behavior of the resulting scalar curvatures, while both being sKlsc with their respective scaling constants, can differ greatly.
}
\end{remark}

\subsection{The degenerate scaling constant}
\label{sKlsc_(2n-1)/n}
Recall that, when $\kappa=\frac{2n-1}{n}$, the sKlsc equation \eqref{sKlsc_PDE_1} degenerates into the first order non-linear autonomous Hamilton-Jacobi equation 
\begin{align}
\label{degenerate_equation_1}
|\nabla f|^2-\frac{1}{n-1}\langle df, \theta\rangle_{\omega}=\frac{S(g)-2(\frac{2n-1}{n})S_C(g)}{2(2n-1)(n-1)}.
\end{align}
A solution sought for here is not required to be positive as this formulation of the sKlsc equation is with respect to the conformal factor $e^{-2f}$.  There are a variety of likely obstructions to the existence of sKlsc metrics with this scaling constant.

By basing \eqref{degenerate_equation_1} at the Gauduchon metric, $g_G$, an integrating, we find that there is an immediate obstruction to the existence of non-K\"ahler sKlsc metrics unless the Gauduchon degree is negative since $\int_M|\nabla f|^2-\frac{1}{n-1}\langle df, \theta\rangle_{\omega}=\int_M|\nabla f|^2\geq0$, because the Lee form is co-closed, and 
\begin{align}
\int_MS(g_G)-2\Big(\frac{2n-1}{n}\Big)S_C(g_G)>0\phantom{=}\text{only if}\phantom{=}\int_M S_C(g_G)<0.
\end{align}

If the conformal class contains a balanced representative, \eqref{degenerate_equation_1} becomes 
\begin{align}
\label{degenerate_equation_2}
|\nabla f|^2=\frac{S(g_b)-2(\frac{2n-1}{n})S_C(g_b)}{2(2n-1)(n-1)}=\frac{(\frac{1-n}{n})S_C(g_b)-\frac{1}{4}|T(g_b)|^2}{(2n-1)(n-1)},
\end{align}
see \cite[Theorem 1.6]{Michelsohn},
and observe that there is an obstruction to the existence of non-K\"ahler sKlsc metrics unless $S_C(g_b)$ is non-positive and not identically zero.  Also, notice that for there to possible exist a solution, $S(g_b)-2(\frac{2n-1}{n})S_C(g_b)$ must vanish at least twice on the manifold.

If the conformal class contains a K\"ahler representative, $g_k$, the degenerate sKlsc equation \eqref{degenerate_equation_2} becomes 
\begin{align}
\label{degenerate_equation_3}
|\nabla f|^2=-\frac{S_C(g_k)}{n(2n-1)},
\end{align}
then there is the further obstruction that $S_C(g_k)$ must equal zero at least at two points on the manifold to the existence of sKlsc metrics.  This is due to the fact that, if there exists a solution $f$, it cannot be constant and must attain distinct maximums and minimums on the manifold.  On the other hand, recall from Section \ref{negative_g_deg}, if $S_C(g_k)$ is non-positive and not identically zero and the conformal class contains a K\"ahler representative, then this degenerate scaling constant $\kappa=\frac{2n-1}{n}$ is the only scaling constant for which there could possibly exist non-K\"ahler sKlsc metrics.

\begin{remark}
\label{hamilton_jacobi_rmk}
{\em
Although there are many compact balanced manifolds with non-positive Chern scalar curvature, and with $S(g_b)-2(\frac{2n-1}{n})S_C(g_b)$  vanishing at least twice on the manifold, it is rare for classical solutions to equations such as \eqref{degenerate_equation_2} to exist, see \cite{Fathi}.  Interestingly, due to the sparseness of solutions for particular values of $\kappa$ limiting to $\frac{2n-1}{n}$, it is even rare for viscosity solutions (with lower regularity) to exist as in \cite{Delanoe}.
}
\end{remark}

\section{Positive Gauduchon degree case}
\label{positive_g_deg}
Recall, from \eqref{kapparestrict}, that the admissible regime of scaling constants when the conformal class has positive Gauduchon degree is $\kappa\in(-\infty,1)$.  For similar reasons as in the negative Gauduchon degree case, the study of the sKlsc question here separates depending upon whether the conformal class does or does not contain a K\"ahler representative.  It is interesting to note, that there is a more stark contrast here between existence results each of these cases than in the negative Gauduchon degree setting.

\subsection{When $[g]$ does not contain a K\"ahler representative}
Consider, first, the sKlsc question in the case that there exists a balanced representative, $g_b$, of the conformal class, but no K\"ahler representative.  Base the family of sKlsc equations at this balanced metric, so the non-self-adjoint term vanish, and they take the form of \eqref{sKlsc_PDE_2}.  Since 
\begin{align}
\int_M S(g_b)-2 S_C(g_b)<0\phantom{=}\text{and}\phantom{=} \int_M S_C(g_b)>0,
\end{align}
there exists an $\epsilon^*>0$ small enough so that, for all $\epsilon\in(0,\epsilon^*)$, the perturbation 
\begin{align}
\int_M S(g_b)-2(1-\epsilon)S_C(g_b)<0.
\end{align}
Therefore, the lowest eigenvalue of the operators $L_{1-\epsilon}$ satisfy $\lambda_0(1-\epsilon)<0$.  Next, observe that any $\kappa<<0$ negative enough ensures that $\int_M S(g_b)-2\kappa \cdot S_C(g_b)>0$, and also that 
\begin{align}
\lim_{\kappa\mapsto-\infty}\frac{(2n-1)(n-1)}{2(n\kappa+1-2n)^2}\Big(S(g_b)-2\kappa\cdot S_C(g_b)\Big)=0.
\end{align}
Therefore, there must exist some range of scaling constants $(-\infty, \kappa^*)\subset(-\infty,1)$ for which the operators $L_{\kappa}$ have lowest eigenvalues $\lambda_0(\kappa)>0$ for all $\kappa\in(-\infty, \kappa^*)$.  
This follows from Theorem \ref{Schrodinger_f_thm} if $S_C(g_b)$ changes sign on the manifold, and directly from examining the Rayleigh quotient if $S_C(g_b)$ is non-negative.  Thus a solution to the sKlsc problem with scaling constant in the regime $\kappa\in(\infty,\frac{2n-1}{n})$ is guaranteed.

\subsection{When $[g]$ contains a K\"ahler representative}
\label{postive_gauduchon_sometimes}
In the case that there does exist a K\"ahler representative, $g_k$, of the conformal class, 
recall that the family of sKlsc equations based at this metric simplifies to $-\Delta \vphi+\frac{(1-\kappa)(2n-1)(n-1)}{(n\kappa+1-2n)^2}S_C(g_k)\cdot \vphi=0$, since for a K\"ahler metric $S(g_k)=2S_C(g_k)$.
By integrating this equation, observe that no non-K\"ahler solution to the sKlsc problem exists if $S_C(g_k)\geq 0$ since $\vphi>0$.  Hence, we now assume that $S_C(g_k)$ changes sign on the manifold.  However, it is not analytically possible to exercise the same control over the behavior of the lowest eigenvalue for scaling constants close to $1$, in other words for $\kappa\in(1-\epsilon,1)$ for small $\epsilon>0$, as in situation where the conformal class does not contain a K\"ahler metric.  This is because here the the potential function will always have positive average.  Examining the multiplier of $S_C(g_k)$ in the potential function, we find that
\begin{align}
\label{max_h(kappa)}
\max_{\kappa<1}\frac{(1-\kappa)(2n-1)(n-1)}{(n\kappa+1-2n)^2}=\frac{2n-1}{4n},
\end{align}
where this maximum is achieved at $\kappa=\frac{1}{n}$.  
Therefore, even if $S_C(g_k)$ changes sign, the sKlsc problem here cannot be formulated in a way, analogous to the proof of the negative Gauduchon degree case above as to satisfy the conditions of Theorem~\ref{Schrodinger_f_thm} since the multiplier on the potential function, which is a function of $\kappa<1$, is bounded.

More specifically, while Theorem \ref{t-schrodinger_cor_II} holds to guarantee the existence of open intervals of scaling constants, $\kappa<1$, near both $1$ and $-\infty$ for which the lowest eigenvalue of
\begin{align}
\label{remarks_operator}
-\Delta+\frac{(1-\kappa)(2n-1)(n-1)}{(n\kappa+1-2n)^2} S_C(g_k)
\end{align}
is strictly positive, since \eqref{max_h(kappa)} is bounded it is not possible to ensure the existence of a member of this family with negative lowest eigenvalue which is necessary in the proof of Theorem~\ref{Schrodinger_f_thm} as it relies on continuity of the lowest eigenvalue in the parameter $\kappa$.  However, Theorem \ref{t-schrodinger_cor_II} can be used as an obstruction to the occurrence of a negative lowest eigenvalue, for if $\frac{2n-1}{4n}$, which recall is the maximum of the multiplier on the potential, satisfies the inequality
\begin{align}
\frac{2n-1}{4n}\leq\frac{\int_M S_C(g_k)}{P||S_C(g_k)||_{\infty}\big(4Vol(M)||S_C(g_k)||_{\infty}+\int_M S_C(g_k)\big)},
\end{align}
then the lowest eigenvalue of this family of Schr\"odinger operators is always positive. 

Note though, this does not preclude the existence of a scaling constant $\kappa<1$ for which the lowest eigenvalue is negative or zero.  In fact, in certain instances we know that non-K\"ahler solutions to sKlsc problem here do exist.  This is seen when $\kappa=\frac{2n-1}{n^2}$, as the sKlsc equation \eqref{sKlsc_PDE_2} becomes
\begin{align}
\label{R-scalar-flat}
-\Delta \vphi+\Big(\frac{n-1}{2n-1}\Big)S_C(g_k)\cdot\vphi=0,
\end{align}
and the existence of a positive solution to this is equivalent to the existence of a Riemannian scalar-flat metric in the conformal class of $g_k$ (recall that $S(g_k)=2S_C(g_k)$ and the real dimension is $2n$).  Note, though, if there does exist a positive solution $\vphi$ to \eqref{R-scalar-flat}, since the conformal factor to obtain the sKlsc metric would be $\vphi^{\frac{2}{n}}$, recall \eqref{conformal_factor_change}, the resulting sKlsc metric would not be Riemannian-scalar-flat as the conformal factor to obtain that metric would be $\vphi^{\frac{2}{n-1}}$.
\begin{remark}
{\em Interestingly, from this along with \eqref{sKlsc_equation_integral}, it is easy to see that if there is a K\"ahler manifold, of complex dimension $n\geq 2$, that is conformal to a non-K\"ahler Riemannian-scalar-flat metric, then the conformal class necessarily has positive Gauduchon degree.
}
\end{remark}
While there do exist solutions to the sKlsc problem in certain instances when the scalar curvature of the initial K\"ahler metric changes sign and satisfies $\int_M S_C(g_k)>0$, it is not at all clear whether solutions should exist for every such manifold, much less the number of solutions that may exist.  In fact, the authors believe that there do exist compact K\"ahler manifolds, satisfying these scalar curvature conditions, which admit no conformal non-K\"ahler sKlsc metrics.  This is loosely reminiscent of the positive Gauduchon degree case for the Chern Yamabe problem \cite{Angella-Simone-Spotti}.

\bibliography{sKlsc_references}

\def\cprime{$'$}
\begin{thebibliography}{10}

\bibitem{Angella-Simone-Spotti}
Daniele Angella, Simone Calamai, and Cristiano Spotti.
\newblock {\em On {C}hern-{Y}amabe Problem}.
\newblock to appear in Math. Res. Lett., 2015.

\bibitem{Melvyn_Berger}
Melvyn~S. Berger.
\newblock On {H}ermitian structures of prescribed nonpositive {H}ermitian
  scalar curvature.
\newblock {\em Bull. Amer. Math. Soc.}, 78:734--736, 1972.

\bibitem{Besse}
Arthur~L. Besse.
\newblock {\em Einstein manifolds}.
\newblock Springer-Verlag, Berlin, 1987.

\bibitem{Dabkowski-Lock_Klsc}
Michael~G. Dabkowski and Michael~T. Lock.
\newblock {\em An equivalence of scalar curvatures on {H}ermitian manifolds}.
\newblock to appear in J. Geom. Anal., 2015.

\bibitem{Dabkowski-Lock_Schrodinger}
Michael~G. Dabkowski and Michael~T. Lock.
\newblock {\em The lowest eigenvalue of {S}chr\"odinger operators on compact
  manifolds}.
\newblock arXiv.org:1508.02755, 2015.

\bibitem{delRio_Simanca}
Heberto del Rio and Santiago~R. Simanca.
\newblock The {Y}amabe problem for almost {H}ermitian manifolds.
\newblock {\em J. Geom. Anal.}, 13(1):185--203, 2003.

\bibitem{Delanoe}
Philippe Delano{\"e}.
\newblock Viscosity solutions of eikonal and {L}ie equations on compact
  manifolds.
\newblock {\em Ann. Global Anal. Geom.}, 7(2):79--83, 1989.

\bibitem{Dolbeault-Esteban-Laptev-Loss}
Jean Dolbeault, Maria~J. Esteban, Ari Laptev, and Michael Loss.
\newblock Spectral properties of {S}chr\"odinger operators on compact
  manifolds: rigidity, flows, interpolation and spectral estimates.
\newblock {\em C. R. Math. Acad. Sci. Paris}, 351(11-12):437--440, 2013.

\bibitem{Evans}
Lawrence~C. Evans.
\newblock {\em Partial differential equations}, volume~19 of {\em Graduate
  Studies in Mathematics}.
\newblock American Mathematical Society, Providence, RI, second edition, 2010.

\bibitem{Fathi}
Albert Fathi.
\newblock Weak {KAM} from a {PDE} point of view: viscosity solutions of the
  {H}amilton-{J}acobi equation and {A}ubry set.
\newblock {\em Proc. Roy. Soc. Edinburgh Sect. A}, 142(6):1193--1236, 2012.

\bibitem{Gauduchon_metric}
Paul Gauduchon.
\newblock Le th\'eor\`eme de l'excentricit\'e nulle.
\newblock {\em C. R. Acad. Sci. Paris S\'er. A-B}, 285(5):A387--A390, 1977.

\bibitem{Gauduchon_scalar}
Paul Gauduchon.
\newblock La {$1$}-forme de torsion d'une vari\'et\'e hermitienne compacte.
\newblock {\em Math. Ann.}, 267(4):495--518, 1984.

\bibitem{Grigor'yan-Nadirashvili-Sire}
Alexander Grigor'yan, Nikolai Nadirashvili, and Yannick Sire.
\newblock A lower bound for the number of negative eigenvalues of
  {S}chr\"odinger operators.
\newblock {\em J. Differential Geom.}, 102(3):395--408, 2016.

\bibitem{Grigor'yan-Netrusov-Yau}
Alexander Grigor'yan, Yuri Netrusov, and Shing-Tung Yau.
\newblock Eigenvalues of elliptic operators and geometric applications.
\newblock In {\em Surveys in differential geometry. {V}ol. {IX}}, Surv. Differ.
  Geom., IX, pages 147--217. Int. Press, Somerville, MA, 2004.

\bibitem{Huybrechts}
Daniel Huybrechts.
\newblock {\em Complex geometry}.
\newblock Universitext. Springer-Verlag, Berlin, 2005.
\newblock An introduction.

\bibitem{Lichnerowicz}
Andr{\'e} Lichnerowicz.
\newblock Spineurs harmoniques.
\newblock {\em C. R. Acad. Sci. Paris}, 257:7--9, 1963.

\bibitem{Liu-Yang_1}
Ke-Feng Liu and Xiao-Kui Yang.
\newblock Geometry of {H}ermitian manifolds.
\newblock {\em Internat. J. Math.}, 23(6):1250055, 40, 2012.

\bibitem{Liu-Yang_2}
Ke-Feng Liu and Xiao-Kui Yang.
\newblock {\em Ricci curvatures on {H}ermitian manifolds}.
\newblock arXiv.org:1404.2481, 2014.

\bibitem{Michelsohn}
M.~L. Michelsohn.
\newblock On the existence of special metrics in complex geometry.
\newblock {\em Acta Math.}, 149(3-4):261--295, 1982.

\bibitem{Schrodinger}
Erwin Schr\"odinger.
\newblock An undulatory theory of the mechanics of atoms and molecules.
\newblock {\em Phys. Rev.}, 28(6):1049,1070, 1926.

\bibitem{Tosatti_non-KahlerCY}
Valentino Tosatti.
\newblock Non-{K}\"ahler {C}alabi-{Y}au manifolds.
\newblock In {\em Analysis, complex geometry, and mathematical physics: in
  honor of {D}uong {H}. {P}hong}, volume 644 of {\em Contemp. Math.}, pages
  261--277. Amer. Math. Soc., Providence, RI, 2015.

\end{thebibliography}

\end{document}